\documentclass[preprint,12pt]{elsarticle}

\usepackage{url}

\usepackage{graphicx,version,stmaryrd,amssymb,amsmath,amsthm}
\usepackage[usenames]{color}

\newcommand{\axname}[1]{\ensuremath{\textit{#1}}}

\newcommand{\SA}{\axname{Signs}}
\newcommand{\sg}{\operatorname{{\mathbf s}}}
\newcommand{\IL}{\axname{IL}}
\newcommand{\Md}{\axname{Md}}
\newcommand{\CC}{\axname{CC}}
\newcommand{\Ril}{\axname{RIL}}
\newcommand{\CEFR}{\axname{CEFR}}
\newcommand{\EFR}{\axname{EFR}}
\newcommand{\AEFR}{\axname{AEFR}}
\newcommand{\SO}{\axname{SO}}
\newcommand{\CR}{\axname{CR}}
\newcommand{\OF}{\axname{OF}}

\newcommand{\CO}{\axname{C\,$0$}}
\newcommand{\SSAV}{\axname{SSAV}}

\newcommand{\Nat}{\ensuremath{\mathbb N}}

\newtheorem{theorem}{Theorem}[section]
\newtheorem{lemma}[theorem]{Lemma}  
\newtheorem{proposition}[theorem]{Proposition}  
\newtheorem{corollary}[theorem]{Corollary}  
\newtheorem*{rev}{Proposition~\ref{efr} (revisited)}  

\newtheorem{remark}[theorem]{Remark}  
\newtheorem{definition}[theorem]{Definition}  

\begin{document}
\begin{frontmatter}
\title{Equations for formally real meadows }

\author[jan]{Jan A.\ Bergstra}
\ead{	J.A.Bergstra@uva.nl}
	
\author[jan]{Inge Bethke\corref{cor1}}
\ead{I.Bethke@uva.nl}

\author[jan]{Alban Ponse}
\ead{A.Ponse@uva.nl}	
	
\cortext[cor1]{Corresponding author}	

\address[jan]{Section Theory of Computer Science, Informatics Institute, University of Amsterdam, The Netherlands}

%\date{}

%\begin{document}

%\maketitle

\begin{abstract}
We consider the signatures  $\Sigma_m=(0,1,-,+, \cdot, \ ^{-1})$ of 
meadows  and $(\Sigma_m, \sg)$ of signed meadows. We give two complete 
axiomatizations of the equational theories of the real numbers with respect to these 
signatures.  In the first case, we extend the axiomatization of zero-totalized fields 
by a single axiom scheme expressing formal realness; the second axiomatization  
presupposes an ordering. 
We apply these completeness results in order to obtain complete axiomatizations of the complex numbers.
\end{abstract}
\end{frontmatter}

\section{Introduction}

The signature $\Sigma_f=(0,1,-, +,\cdot)$ of \emph{fields} has two constants 0 and 
1, a unary function $-$, and two binary functions $+$ and $\cdot$. The first-order 
theory of fields is given by the axioms of commutative rings 
(see Table~\ref{commring}) and two additional axioms, namely
\begin{eqnarray*}
	0&\neq & 1,\\
	x\neq 0   &\rightarrow &\exists y \ x\cdot y=1.
\end{eqnarray*}

\begin{table}[htbp]
\centering
\hrule
\begin{align*}
\qquad(x+y)+z &= x + (y + z)\qquad\\
	x+y     &= y+x\\
	x+0     &= x\\
	 x + (-x)   &= 0\\
	(x \cdot y) \cdot  z &= x \cdot  (y \cdot  z)\\
	x \cdot  y &= y \cdot  x\\
	1\cdot x &= x \\
	x \cdot  (y + z) &= x \cdot  y + x \cdot  z
	\end{align*}
\hrule
\caption{The set  $\CR$ of axioms for commutative rings}
\label{commring}
\end{table}

A field $F$ is said to be \emph{ordered} if there  exists  a subset 
$F^{>0}\subseteq F$|the set of positive elements in $F$|such that $F^{>0}$ is 
closed under addition and multiplication, and $F$ is the disjoint union of 
$F^{>0}$, $\{0\}$, and $\{-a\mid a\in F^{>0}\}$. Then $F$ is totally ordered if we 
define $a>b$ to mean $a-b\in F^{>0}$. Moreover, if $a>b$, then $a+c >b+c$ for every 
$c$ and $a\cdot  c>b\cdot  c$ for every $c\in F^{>0}$. 
The theory of  ordered fields is formulated over the signature 
$\Sigma_{\textit{of}}=(0,1,-, +,\cdot, <)$. It has all the field axioms and, in 
addition, the axioms for  a total ordering  that is compatible with the field 
operations given in Table~\ref{orderedfield}. 

In 1927, the theory of ordered fields grew into the Artin-Schreier theory of 
ordered fields and formally real fields. 
\begin{definition}\label{formalreal}
A field $F$ is called \emph{formally real} if $-1$ is not a sum of squares in $F$.
\end{definition}
A main result of the Artin-Schreier theory  (see e.g.~\cite{J80}) states:

\begin{proposition}\label{sum_of_squares}
Let $F$ be an arbitrary field. $F$ is formally real if and only if for all 
$n\geq0$ and all $x_0, \ldots , x_n\in F$ we have 
\[
 \Sigma_{i=0}^{n} x^2_i = 0 \Rightarrow x_0= \cdots =x_n=0.
\]
\end{proposition}
Formally real fields can therefore be axiomatized by the following infinite list of 
axioms, one for each $n\geq 0$,
\[
\forall x_0 \forall x_1 \cdots \forall x_n ( x_0\cdot x_0 + \cdots  + x_n \cdot x_n 
= 0 \rightarrow (x_0 = 0\  \wedge\ \cdots \ \wedge\ x_n = 0)).
\]
A formally real field has no defined order relation. However, it is always possible 
to find an ordering (and often more) that will change a formally real field into an 
ordered field. One can view a formally real field as an ordered field  where the 
ordering is not explicitly given. The fields of rational numbers $\mathbb{Q}$ and of 
real numbers $\mathbb{R}$ are examples.

\begin{table}[htbp]
\centering
\hrule
\begin{align*}
\label{OF1}
\tag{OF1}
x\neq 0 &\rightarrow( x<0 \ \vee\ 0<x )&  \\
x < y& \rightarrow\neg ( y < x \ \vee\ x=y)& \tag{OF2}\label{OF2}\\
x<y & \rightarrow  x+z < y+z &\tag{OF3}\label{OF3}\\
x<y \ \wedge\ 0<z & \rightarrow  x\cdot z <y \cdot z&\tag{OF4}\label{OF4}\
\end{align*}
\hrule
\caption{The set  $\OF$ of axioms  for ordered fields}
\label{orderedfield}
\end{table}

Since the signature of fields does not  include a multiplicative inverse, the axiom 
for the inverse is  not universal, and therefore a substructure of a field closed 
under multiplication is not always a field. This can be remedied by adding a unary 
inverse operation $\ ^{-1}$ to the language.
In~\cite{BT07} \emph{meadows} were defined as members of a variety specified by 
equations. A {meadow} is a commutative ring equipped with
a total unary operation $\ ^{-1}$
named \emph{inverse} that satisfies $0^{-1}=0$. Every field $F$ can be expanded to a 
meadow (or \emph{zero-totalized field}) $F_0$ after making  the inverse operator 
total by $0^{-1}=0$. Thus $\mathbb{Q}_0$, $\mathbb{R}_0$ and $\mathbb{C}_0$ are 
meadows|the meadows of the rational, real and complex numbers, respectively.

An advantage of meadows over working with the signature of fields is that it 
facilitates formal reasoning without requiring the use of
either a logic of partial operations or a three valued logic. We will exploit this 
advantage and prove two completeness results the statement and meaning of which are 
accessible to all people who have been exposed to real numbers in elementary 
mathematics. We will view the real numbers as a formally real meadow 
which can be equipped with an ordering that is encoded in a sign function.
We will prove that all valid equations over the meadow of reals are derivable from 
the axioms of meadows 
plus an axiom scheme expressing formal realness; the valid equations over the signed 
meadow of the reals follow from the axioms of meadows together with the axioms for 
the sign function.

The remainder of this paper is organized as follows. The next section comprises 
preliminaries and a digression on meadows of characteristic 0.
In Section~\ref{sec:3}  we 
introduce formally real meadows and provide two  axiomatizations with accompanying 
completeness results.
Our first completeness result is
\begin{equation}
\Md +\EFR\vdash s=t \quad \text{ if and only if } \quad\mathbb{R}_0\models 
s=t \tag{\dag}
\end{equation}
where $\Md$ is a finite equational axiomatization of meadows and $\EFR$ is an 
equational axiom scheme expressing formal realness. We obtain $(\dag)$ as an 
immediate consequence of the Artin-Schreier Theorem and Tarski's theorem on 
quantifier elimination for real closed fields.
Moreover, we introduce signed meadows and give a finite axiomatization of formally 
real meadows expanded by a sign function. Our second completeness result with 
respect to formally real meadows is 
\begin{equation}
\Md +\SA\vdash s=t \quad\text{ if and only if } \quad (\mathbb{R}_0,\sg)\models 
s=t \tag{\ddag}
\end{equation}
with $\SA$ a finite set of axioms for the sign function $\sg$. 
Also $(\ddag)$ relies on both the Artin-Schreier Theorem and Tarski's theorem.
In the last section, we apply these completeness results and obtain complete axiomatizations of the meadow of complex numbers.

\section{Preliminaries}
\label{sec:2}
In~\cite{BT07} \emph{meadows}
were defined as the members of a variety specified by twelve equations. 
However, in~\cite{BHT}
it was established that the ten equations in Table~\ref{Md}
imply those used in~\cite{BT07}. Summarizing,
a {meadow} is a commutative ring with unit equipped with
a total unary operation $\ ^{-1}$
named \emph{inverse} that satisfies the two equations
\begin{align*}
  (x^{-1})^{-1} &= x, \\
  x\cdot(x \cdot x^{-1}) &= x.& \tag{\Ril}
\end{align*}
Here \Ril\ abbreviates 
\emph{Restricted Inverse Law}.  We write \Md\ for the set of
axioms in Table~\ref{Md} and write $\Sigma_m=(\Sigma_f, ^{-1})$ for the signature 
of meadows. From the axioms in \Md\ the following identities are derivable:
\begin{align*}
	0^{-1}  &= 0,\\
	(-x)^{-1} &= -(x^{-1}),\\
	(x \cdot  y)^{-1} &= x^{-1} \cdot  y^{-1},
\\
0\cdot x  &= 0,\\
	x\cdot -y &= -(x\cdot y),\\
	-(-x)     &= x.
\end{align*}

\begin{table}[t]
\centering
\hrule
\begin{align*}
	(x+y)+z &= x + (y + z)\\
	x+y     &= y+x\\
	x+0     &= x\\
	x+(-x)  &= 0\\
	(x \cdot y) \cdot  z &= x \cdot  (y \cdot  z)\\
	x \cdot  y &= y \cdot  x\\
	1\cdot x &= x \\
	x \cdot  (y + z) &= x \cdot  y + x \cdot  z\\
	(x^{-1})^{-1} &= x \\
	x \cdot (x \cdot x^{-1}) &= x
\end{align*}
\hrule
\caption{The set \Md\ of axioms for meadows}
\label{Md}
\end{table}

\begin{table}[b]
\centering
\hrule
\begin{align*}
x-y&=x+(-y)\notag\\
1_x&=x\cdot x^{-1} \notag\\
0_x&=1-1_x \notag
\end{align*}
\hrule
\caption{The derived operators subtraction, pseudo ones and pseudo zeros}
\label{derived}
\end{table}

We often use the derived operators \emph{subtraction}, \emph{pseudo ones} and 
\emph{pseudo zeros} given in Table~\ref{derived}.
Pseudo constants enjoy a couple of nice properties which are listed in the appendix. 
The most prominent are $0_0 = 1_1 = 1$, $0_1 = 1_0 = 0$ and 
\[
0_t + 1_t = 1 \text{ and } 0_t \cdot 1_t = 0.
\]
for all terms $t$.
In the remainder we shall tacitly assume that a meadow has subtraction and pseudo 
constants.
Moreover, we freely use numerals $\underline{n}$|defined by $\underline{0}=0$, 
$\underline{1}=1$ and $\underline{n+1}=\underline{n} + 1$ for $n\geq 1$|and 
exponentiation with constant integer exponents. 

The term \emph{cancellation meadow} is introduced in~\cite{BPvdZ} 
for a zero-totalized field that satisfies the so-called 
``cancellation axiom"
\[
x \neq 0 ~\wedge~ x\cdot y = x\cdot z ~\rightarrow~ y=z.\]
An equivalent version of the cancellation axiom that we shall
further use in this paper is the
\emph{Inverse Law} (\IL), i.e., the conditional axiom
\begin{align*}
  x\neq 0 ~\rightarrow~ x\cdot x^{-1}=1.
 \tag{\IL}
\end{align*}
So \IL\ states that there are no proper zero divisors. 
(Another equivalent formulation of the cancellation property is 
$x\cdot y=0~\rightarrow~x=0\ \vee\ y=0$.) Paradigm cancellation meadows are 
$\mathbb{Q}_0$, $\mathbb{R}_0$ and $\mathbb{C}_0$. However, there also exist meadows 
with proper zero devisors and infinite meadows with characteristic $0$
different from $\mathbb{Q}_0$, $\mathbb{R}_0$ and $\mathbb{C}_0$.  For example, 
in~\cite{BRS09} it is proved that $\mathbb{Z}/n\mathbb{Z}$ with elements 
$\{0, 1, \ldots , n-1\}$ where arithmetic is performed modulo $n$ is a meadow if 
$n$ is squarefree, i.e.\
$n$ is the product of pairwise distinct primes. Thus $\mathbb{Z}/10\mathbb{Z}$ is a 
meadow where $2\neq 0\neq  5$ but $2\cdot 5=0$. The existence of an infinite 
non-cancellation meadow is shown in the following theorem.

\begin{theorem}
There exists a non-cancellation meadow $M$ of characteristic 0 which does not
have $\mathbb{Q}_0$ as a subalgebra.
\end{theorem}

\begin{proof}
Choose a new constant symbol $a$. For $k\in \mathbb{N}$ let
\[
E_k = \{a \neq 0\} \cup \{\underline{n}\neq 0 \mid n \in \mathbb{N}, 0 < n < k\} 
\cup  \{\underline{2}\cdot a = 0\} \cup \Md.
\]
Moreover, choose a prime $p \neq 2$ exceeding $k$ and interpret $a$ in 
$\mathbb{Z}/2p\mathbb{Z}$ by $p$. Then $\mathbb{Z}/2p\mathbb{Z}\models E_k$.
It follows that $E_k$ is consistent and therefore $E = \bigcup_{k=1}^\infty E_k$ 
is consistent by the compactness
theorem. Let $M$ be a model for $E$. Then
\begin{enumerate}
\item $M$ is a meadow, since $\Md \subseteq E$,
\item $M$ has characteristic 0, since $M \models \underline{n}\neq 0$ for all 
$n\in \mathbb{N}$ with $n \neq 0$,
\item $M$ is not a cancellation meadow, since $M\models  \underline{2} \neq 0$, 
$M \models a \neq  0$, but $M \models \underline{2}\cdot a = 0$, and
\item $M \models \underline{2}\cdot \underline{2}^{-1}\neq  1$, for otherwise
\[
a = 1 \cdot a =  \underline{2}\cdot \underline{2}^{-1} \cdot a 
= \underline{2}^{-1}\cdot \underline{2} \cdot a= \underline{2}^{-1} \cdot 0 = 0;
\]
hence  $\mathbb{Q}_0$ is not a subalgebra of $M$.
\end{enumerate}
\end{proof}

In~\cite{BBP2013}, we proved a finite basis result for the equational theory of 
cancellation meadows. This result is formulated in a generic way so that it can be 
applied to any expansion of a meadow that satisfies the propagation properties 
defined below.

\begin{definition}
Let $\Sigma$ be an extension of $\Sigma_{m}$ and  $E\supseteq \Md$.
\begin{enumerate}
\item
$(\Sigma,E)$ has the \emph{propagation property for pseudo ones} if for
each pair of $\Sigma$-terms $t,r$ and context $C[~]$,
\[E\vdash 1_t\cdot C[r]=1_t\cdot C[1_t\cdot r].\]
\item
$(\Sigma,E)$ has the \emph{propagation property for pseudo zeros} if for
each pair of $\Sigma$-terms $t,r$ and context $C[~]$,
\[E\vdash 0_t\cdot C[r]=0_t\cdot C[0_t\cdot r].\]
\end{enumerate}
\end{definition}
Preservation of these propagation properties admits the following 
nice result:
\begin{theorem}
[Generic Basis Theorem for Cancellation Meadows]
\label{st}
If $\Sigma\supseteq \Sigma_{m}$, $E\supseteq \Md $ and $(\Sigma,E)$ 
has the
pseudo one propagation property and the pseudo zero propagation property,
then 
\[
E\vdash s=t \quad\text{ if and only if }\quad E + \IL \models s=t
\]
for all $s,t\in \Sigma$.
\end{theorem}
\begin{proof} 
This is Theorem 3.1 of~\cite{BBP2013}.
\end{proof}

Meadow terms can be represented in a particular standard way. 
\begin{definition}
A term $P$ over $\Sigma_m$ is a \emph{Standard Meadow Form
(SMF)} if, for some $n\in\Nat$, $P$ is an \emph{SMF of level $n$}.
SMFs of level $n$ are defined as follows:
\begin{description}
%\\[1mm]
%\item[\textit{SMF of level $\mathbf 0$}] each expression of the form $s\cdot t^{-1}$ 
\item SMF of level $0$: each expression of the form $s\cdot t^{-1}$ 
with  $s$ and $t$ $\Sigma_f$-terms,
%\\[.6mm]
%\item[\textit{SMF of level $\mathbf {n+1}$}] each expression of the form
\item SMF of level ${n+1}$: each expression of the form
\[0_t\cdot P+1_t\cdot Q\]
with  $t$  a $\Sigma_f$-term and $P$ and $Q$  SMFs of 
level $n$.
\end{description}
\end{definition}
\begin{theorem}
\label{SMF}
For each term $s$ over $\Sigma_{m}$ there exists  an SMF $s'$
with the same variables such that  $\Md\vdash s=s'$.
\end{theorem}
\begin{proof}
This is Theorem 2.1 of~\cite{BBP2013}.
\end{proof}

It follows that every  meadow equation has a first-order representation over the 
signature of  fields. Since we will apply the first-order representation solely in 
the context of cancellation meadows we may freely use  $\IL$.
\begin{corollary}\label{phi}
For each equation $s=t$ over $\Sigma_{m}$ there exists a quantifier-free first-order
formula $\phi(s,t)$ over $\Sigma_f$ with the same variables such that 
\[\Md + \IL\vdash s=t \leftrightarrow \phi(s,t).\]
\end{corollary}

\begin{proof}
By the preceding theorem we may assume that both $s$ and $t$ are SMF's. We employ 
induction on the levels $n,m$ of $s$ and $t$, respectively.
\begin{enumerate}
\item 
$n=0$: Then there are terms $s_1, s_2$ over $\Sigma_{\textit{f}}$ such that 
$s=s_1\cdot s_2^{-1}$. 
\begin{enumerate}
\item $m=0$: Then there are terms $t_1, t_2$ over $\Sigma_f$ such that 
$t=t_1 \cdot t_2^{-1}$. Since $s=t$ we have $s_2\cdot t_2\cdot (s_1
\cdot t_2 - t_1\cdot s_2) =0$ and hence $s_2=0$ or $t_2=0$ or $s_1\cdot t_2 =
t_1\cdot s_2$. Thus we can take 
\begin{eqnarray*}
\phi(s,t)&\equiv &(s_2=0 \rightarrow (t_1=0 \ \vee\ t_2=0))\  \wedge\ \\
&& (t_2=0 \rightarrow (s_1=0 \ \vee\ s_2=0)) \ \wedge\ \\
&&(s_2\cdot t_2\neq 0 \rightarrow s_1\cdot t_2=t_1\cdot s_2).
\end{eqnarray*}
\item $m=k+1$: Then $t= 0_{t'}\cdot P + 1_{t'}\cdot Q$ for some 
$\Sigma_f$ term $t'$ and SMFs of level $k$. Observe that if $t'=0$, then 
$1_{t'}=0$ and $0_{t'}=1$ and therefore $s=P$. Likewise, if $t'\neq 0$, then $1_{t'}
=1$ and $0_{t'}=0$ and therefore $s=Q$. Thus we can take
\[
\phi(s,t)\equiv (t'=0 \rightarrow \phi(s,P)) \ \wedge\ (t'\neq 0 
\rightarrow \phi(s,Q)).
\]
\end{enumerate}
\item $n=l+1$: Then $s= 0_{s'}\cdot P + 1_{s'}\cdot Q$ for some 
$\Sigma_f$-term $s'$ and SMFs $P, Q$ of level $k$. 
Now we argue as in case 1(b) and take
\[
\phi(s,t)\equiv (s'=0 \rightarrow \phi(P,t)) \ \wedge\ (s'\neq 0 
\rightarrow \phi(Q,t)).
\]
\end{enumerate}
\end{proof}

\begin{table}
\centering
\hrule
\begin{align*}
	1_{\underline{n+1}}&=1&\tag{C0$_n$}
\end{align*}
\hrule
\caption{The set of axioms  $\CO$ for meadows of characteristic 0}
\label{CO}
\end{table}

We  use pseudo ones to give an infinite axiomatization of meadows of characteristic 
0 in Table~\ref{CO}.  As a corollary to Theorem~\ref{st} we have:

\begin{corollary}\label{basis0}
Let $s,t\in \Sigma_{m}$. Then
\[
\Md  +\CO \vdash s=t \quad\text{ if and only if } \quad \Md  +\CO + \IL \models s=t.
\]
\end{corollary}
\begin{proof}
Since ($\Md$, $\Sigma_{m}$) has the propagation properties (see Corollary 3.1 
of~\cite{BBP2013}), the derived operators also share these properties. 
We therefore can apply Theorem~\ref{st}.
\end{proof}

An \emph{extension field} $\widehat{F}$ of a field $F$ is any field containing 
$F$ as a subalgebra. 
%An element $x$ of $\widehat{F}$ is said to be 
%\emph{algebraic over $F$} if there exists a non-zero polynomial $P$ with 
%coefficients from  $F$ such that $P(x)=0$. 
An extension field $\widehat{F}$ of 
$F$ is said to be \emph{algebraically closed} if every non-trivial polynomial has a root in $F$. 
Every field has an algebraic closure. In 1951, Tarski~\cite{T51} proved that the 
theory of algebraically closed fields in the first-order language over the signature 
$(0,1, -, +, \cdot )$ admits elimination of quantifiers. The most important model 
theoretic consequences hereof are the completeness of the theory of algebraically 
closed fields of a given characteristic|in particular, the  theory of algebraically 
closed fields of characteristic 0 coincides with the theory of the complex numbers 
over the signature $(0,1,-,  +, \cdot )$|and its decidability 
(see e.g.~\cite{CK90, P00}).

A reduct of an algebraic structure is obtained by omitting some of the operations 
and relations of that structure. The converse of a reduct is an expansion. In the 
sequel, we will write $M|\Sigma'$ for the reduct of the $\Sigma$-algebra $M$ to 
$\Sigma'\subseteq \Sigma$ and $(M,\diamond)$ for the expansion of $M$ by an 
operation or relation $\diamond$. In particular, we will write $M|\Sigma_f$ 
for the reduct of the meadow $M$ to the signature of fields, and $(M,\sg)$ and 
$(M,<)$ for the expansion of the meadow $M$ with a sign function and order 
relation, respectively.

\begin{theorem}\label{complete0}
Let $s,t$ be $\Sigma_{m}$-terms. Then 
\[
\Md + \CO\vdash s=t \quad  \text{ if and only if } \quad
\mathbb{C}_0\models s=t.
\]
\end{theorem}
\begin{proof} 
Because  $\mathbb{C}_0$ constitutes a meadow of characteristic 0, soundness is 
immediate. 

Assume that $s=t$
is not derivable from $\Md +\CO$. 
By the previous corollary, there  exists a cancellation meadow $M$ of characteristic 
0 with $M\not \models  s=t$. Since $M$ is a field, it has an algebraic closure, say, 
$\widehat{M}$.
Then $\widehat{M}\not\models s= t$ and hence 
$\widehat{M}|\Sigma_f\not \models  \phi(s,t)$ where $\phi(s,t)$ is the 
$\Sigma_f$-representation given in Corollary~\ref{phi}. 
Therefore $\mathbb{C}_0|\Sigma_f\not \models  \phi(s,t)$  by completeness and 
thus $\mathbb{C}_0\not \models s= t$.
\end{proof}

\begin{remark}
Initial algebras provide standard models of equational specifications. They contain 
only elements that can be constructed from those appearing in the specification, and 
satisfy only those closed equations that appear in the specification or are logical 
consequences of them. It is easy to see that $\Md + \CO$ constitutes an initial 
algebra specification of the rational numbers. See also~\cite{BT07}.
\end{remark}

\begin{remark}
In a way very similar to Theorem~\ref{complete0} one can show that 
$\Md +\{\underline{p}=0\}$
is a complete axiomatization of algebraically closed 
fields of characteristic $p$.
\end{remark}

\section{Formally real meadows}
\label{sec:3}
Our first axiomatization of formal realness is the infinite set of axioms given in
Table~\ref{EFR}.

\begin{table}
\centering
\hrule
\begin{align}
	0_{x_0^2 + \cdots +x_n^2}\cdot x_0&=0&\tag{EFR$_n$}\label{EFRn}
\end{align}
\hrule
\caption{The set $\EFR$ of axioms for formally real meadows}
\label{EFR}
\end{table}
\begin{definition}
A meadow $M$ is called formally real if $M\models \EFR$.
\end{definition}
Observe that \eqref{EFRn} is an equational representation of the conditional axiom
\[
\CEFR_n\equiv x_0^2 + \cdots x_n^2 = 0 \rightarrow x_0=0.
\]

\begin{proposition}
For all $n\geq 0$, $\Md + \EFR\vdash \CEFR_n$.
\end{proposition}

\begin{proof}
Suppose $x_0^2 + \cdots x_n^2 = 0$. Then $0_{x_0^2 + \cdots x_n^2}= 0_0=1$. 
Hence $x_0=0$ by \eqref{EFRn}.
\end{proof}

Again $\mathbb{Q}_0$ and $\mathbb{R}_0$ are formally real meadows; however, 
$\mathbb{C}_0$ is not formally real.

As a corollary to Theorem~\ref{st} we have

\begin{corollary}\label{basis}
Let $s,t\in \Sigma_{m}$. Then
\[
\Md  +\EFR \vdash s=t \quad \text{ if and only if } 
\quad \Md  +\EFR + \IL \models s=t.
\]
\end{corollary}
\begin{proof}
This follows again from the fact that  ($\Md $, $\Sigma_{m}$) has the propagation 
properties (see Corollary 3.1 of~\cite{BBP2013}).
\end{proof}

\begin{remark}
As an alternative we could have chosen the axiom scheme \AEFR\  given in 
Table~\ref{AEFR} as an axiomatization of formal realness. Observe that \eqref{AEFRn} 
expresses the fact that \[\Sigma_{i=0}^n x_i^2 \neq -1\] which is somewhat closer to 
the definition of formal realness found in the literature 
(see also Definition~\ref{formalreal}). In the following theorem we prove 
that \EFR\ and \AEFR\ are indeed equivalent in the context of meadows.
\begin{table}[ht]
\centering
\hrule
\begin{align}
	1_{1+ x_0^2 + \cdots +x_n^2}&=1&\tag{AEFR$_n$}\label{AEFRn}
\end{align}
\hrule
\caption{The set $\AEFR$ of alternative axioms for formally real meadows}
\label{AEFR}
\end{table}
\begin{theorem}\label{alternative}
Let $M$ be a meadow and $s,t\in \Sigma_{m}$. Then 
\[\Md + \EFR \vdash s=t \quad \text{ if and only if } \quad \Md+ \AEFR\vdash s=t.\]
\end{theorem}
\begin{proof}
Observe that Corollary~\ref{basis} also holds for \AEFR\ instead of \EFR. 
It therefore suffices to prove that 
\[ 
M\models \EFR \quad \text{ if and only if } \quad M\models \AEFR
\] 
for every cancellation meadow $M$. The claim now follows from 
Proposition~\ref{sum_of_squares} and the fact that a cancellation meadow is a field.
\end{proof}
\end{remark}

A \emph{real closed} field  $F$ is a  formally real field in which every positive 
element has a square root and, moreover, every polynomial of odd degree in one 
indeterminate with coefficients in $F$ has a root. In~\cite{T51}, Tarski also proved 
quantifier elimination for real closed fields in the first-order language over the 
signature $(0,1, -, +, \cdot, <)$. As a consequence, the theory of real closed 
fields is complete |in particular, the theory of real closed fields coincides with 
the theory of the real numbers  over the signature $(0,1, +, \cdot, -, <)$|and is 
decidable (see e.g.~\cite{CK90, P00}).

If $(F,<)$ is an ordered field, the \emph{Artin-Schreier Theorem}~\cite{AS27} 
states that $F$ has an algebraic extension, called the \emph{real closure} 
$\widehat{F}$ of $F$, such that $\widehat{F}$ is a real closed field whose ordering 
is an extension of the given ordering $<$ on $F$ and is unique up to isomorphism 
of fields. The classical proof of this theorem relies on the use of Zorn's Lemma 
which is equivalent to the axiom of choice. However, in 1991 Sander~\cite{S91} 
showed the existence of real closures of ordered fields in Zermelo-Fr\"ankel 
axiomatic set theory.

Our first completeness result for formally real meadows follows the proof of 
Theorem~\ref{complete0} and, in addition, relies on the Artin-Schreier Theorem 
and Tarski's theorem for real closed fields.
\begin{theorem}\label{complete1}
Let $s,t$ be $\Sigma_{m}$-terms. Then 
\[
\Md + \EFR\vdash s=t \quad \text{ if and only if } \quad
\mathbb{R}_0\models s=t.
\]
\end{theorem}
\begin{proof} 
Because the meadow of real numbers  is formally real, soundness is immediate. 

Assume that $s=t$
is not derivable from $\Md +\EFR$. 
From Corollary~\ref{basis} we get the existence of  a formally real  cancellation 
meadow
$M$ with $M\not \models  s= t$. Since $M$ is formally real, it is orderable by, 
say, $<$. By the Artin-Schreier Theorem $(M, <)$ has  the real closure 
$(\widehat{M}, <)$.
Since  $(\widehat{M},<)\not \models s=t$,  we find that  
$(\widehat{M}, <)|\Sigma_f\not \models  \phi(s,t)$ where $\phi(s,t)$ is the 
first-order representation of $s=t$ given in  Corollary~\ref{phi}.
Since the theory of real closed fields coincides  with the theory of the real 
numbers  over the signature $(0,1,-,  +, \cdot, <)$, we can conclude that 
$\mathbb{R}_0|\Sigma_f\not \models \phi(s,t)$ and therefore 
$\mathbb{R}_0\not \models  s= t$.
\end{proof}
The theory of $\Md +\EFR$ or, equivalently, the theory of  $\Md +\AEFR$ cannot be finitely axiomatized.
\begin{theorem}
$\Md +\AEFR$ has no finite axiomatization.
\end{theorem}
\begin{proof}
Suppose $\Md +\AEFR$ has  a finite axiomatization. Then by pure logic, there must exist $n\in \Nat$ such that $\Md +\AEFR_n$ is an axiomatization of $\Md +\AEFR$. In order to obtain a contradiction, we construct a cancellation meadow $K_n$ with
\begin{enumerate}
\item every sum of $n+1$ squares is unequal to -1 (and thus $K_n$ satisfies $\AEFR_n$), and 
\item  -1 is the sum of $m+1$  squares  for some $m> n$ (and thus $K_n$ does not satisfy $\AEFR_{m}$).
\end{enumerate}
The \emph{Pythagoras number} of a field is the smallest positive integer $p$ such that every sum of squares is a sum of $p$ squares. In \cite{P78}, Prestel showed that for every natural number $k$ there exists  an ordered subfield $R_k$ of $\mathbb{R}$ with Pythagoras number $k$. Now let for $k=(n+2)(n+1)+1$ $R_{k}$ be such an ordered subfield of $\mathbb{R}$ with Pythagoras number $k$ and choose $r=r_1^2 + \cdots + r_{k}^2$ in $R_{k}$ that cannot be written as the sum  of $(n+2)(n+1)$ squares. Let $K_n=R_{k}(\sqrt{-r})$ be the zero-totalized field extension of $R_{k}$ with the root of $-r$ and observe that in $K_n$
\begin{enumerate}
\item every sum of $n+1$ squares is unequal to -1: Suppose $s_1^2  +\cdots +s_{n+1}^2=-1$. Viewing $K_n$ as a quotient ring we can write every $s_i$ as $a_i+b_i\cdot \sqrt{-r}$ with $a_i, b_i \in R_{k}$. It follows that
\[
-1 = \Sigma_{i=1}^{n+1} (a_i^2 - b_i^2r) + (\Sigma_{i=1}^{n+1} 2a_ib_i)\cdot \sqrt{-r}.
\]
Therefore $-1 =  \Sigma_{i=1}^{n+1} (a_i^2 - b_i^2r)$. Since $R_k$ is an ordered field, -1 cannot be the sum of squares. So $b_1^2 + \cdots +b_{n+1}^2\neq 0$.  Hence
\[
r=(1^2 + a_1^2 + \cdots + a_{n+1}^2)(b_1^2 + \cdots +b_{n+1}^2)^{-1}.
\]
Thus
\[
r=(1^2 + a_1^2 + \cdots + a_{n+1}^2)(b_1^2 + \cdots +b_{n+1}^2)(b_1^2 + \cdots +b_{n+1}^2)^{-2}.
\]
Whence $r$ can be written as a sum of $(n+2)(n+1)$ squares contrary to the particular choice of $r$.
\item -1 is the sum of $k$ squares: Observe that $r\neq 0$, since it cannot be written as the sum of $(n+2)(n+1)$ squares. Therefore
\[
\begin{array}{rcl}
-1 =r\cdot (\sqrt{-r})^{-2} &=& r_1^2\cdot (\sqrt{-r})^{-2} + \cdots + r_{k}^2\cdot (\sqrt{-r})^{-2} \\
&=&( r_1\cdot \sqrt{-r}^{-1})^2 + \cdots + (r_{k}\cdot \sqrt{-r}^{-1})^2.
\end{array}
\]
\end{enumerate}
\end{proof}
A finite axiomatization of formal realness is obtained by 
extending the signature $\Sigma_m$ of 
meadows with the unary sign (or signum) function $\sg(x)$.
We write $\Sigma_{\textit{fs}}$, $\Sigma_{\textit{ms}}$ for these extended 
signatures, so $\Sigma_{\textit{fs}} = (\Sigma_f, \sg)$ and $\Sigma_{\textit{ms}}
 = (\Sigma_m, \sg)$.
The sign function $\sg(x)$
presupposes an ordering on its domain and
is defined by
\[\sg(x)=\begin{cases}
-1&\text{if }x<0,\\
0&\text{if }x=0,\\
1&\text{if }x>0.
\end{cases}\]

We define the sign function in an equational manner
by the set \SA\ of axioms 
given in Table~\ref{t:a}.
First, notice that by $\Md $ and axiom \eqref{S1}
(or axiom \eqref{S2}) we find 
\[\sg(0)=0\quad\text{and}\quad\sg(1)=1.\]
Then, observe that in combination with the inverse law \IL, 
axiom \eqref{S6} is an equational representation of the conditional axiom
\[\sg(x)=\sg(y)~\rightarrow ~\sg(x+y)=\sg(x).\]
From $\Md$ and axioms
\eqref{S3}--\eqref{S6} one 
can easily compute $\sg(t)$ for any closed term $t$.
Some more consequences of the $\Md  +\SA$ axioms are these (see~\cite{BBP2013}):
\begin{align}
\sg(x^2)&=1_x& \tag{S7}\label{S7}\\
\sg(x^3)&=\sg(x)&\tag{S8}\label{S8}\\
1_x\cdot\sg(x)&=\sg(x)& \tag{S9}\label{S9}\\
\sg(x)^{-1}&=\sg(x)&\tag{S10}\label{S10}\\
\sg(\sg(x))&=\sg(x)&\tag{S11}\label{S11}
\end{align}

\begin{table}
\centering
\hrule
\begin{align}
\sg(1_x)&=1_x&\tag{S1}\label{S1}\\
\sg(0_x)&=0_x&\tag{S2}\label{S2}\\
\sg(-1)&=-1&\tag{S3}\label{S3}\\
\sg(x^{-1})&=\sg(x)&\tag{S4}\label{S4}\\
\sg(x\cdot y)&=\sg(x)\cdot \sg(y)&\tag{S5}\label{S5}\\
0_{\sg(x)-\sg(y)}\cdot (\sg(x+y)-\sg(x))&=0&\tag{S6}\label{S6}
\end{align}
\hrule
\caption{The set \SA\ of axioms for the sign function}
\label{t:a}
\end{table}
Expanding $\mathbb{Q}_0$ and $\mathbb{R}_0$ in the obvious way, we obtain signed 
meadows $(\mathbb{Q}_0, \sg)$ and $(\mathbb{R}_0, \sg)$. 

Signed meadows need not be cancellation meadows.   This can easily be seen as 
follows. Expand the signature of signed meadows by two new constants $a, b$ and 
consider the equational theory $T=\Md  + \SA +\{a \cdot b=0\}$. 
We obtain a model for $T$ by interpreting in $(\mathbb{R}_0, \sg)$ $a$ by 0 and $b$ 
by 1 and construct a second model by interchanging the interpretations of $a$ and 
$b$. It follows that the initial algebra of $T$ satisfies neither $a=0$ nor $b=0$. 
However, $\Md  \cup \SA$ has the propagation properties for pseudo ones and zeros 
(see Corollary~5.1 of~\cite{BBP2013}). 
We therefore have the following finite basis result for signed cancellation meadows.
\begin{corollary}\label{basis2}
Let $s,t\in \Sigma_{ms}$. Then
\[
\Md  +\SA\vdash s=t \quad\text{ if and only if } \quad \Md  +\SA+ \IL \models s=t.
\]
\end{corollary}
Signed  meadows are  formally real.

\begin{proposition}\label{efr}
For every $n\geq 0$,  $\Md + \SA \vdash \EFR_n$.
\end{proposition} 

\begin{proof}
We give a semantic proof; a syntactic proof can be found in the appendix. By 
Corollary~\ref{basis2}
it suffices to prove that $M\models  \EFR_n$ for all signed cancellation meadows. 
Thus assume
$M \models \Md  + \SA + \IL$. Observe that if $a^2_0+\cdots + a^2_n \neq 0$ then 
$1_{a^2_0+\cdots + a^2_n}=1$
and hence $0_{a^2_0+\cdots + a^2_n}\cdot a_0=0$.
We can therefore assume that $a^2_0+\cdots + a^2_n =0$ and prove 
$M \models  \EFR_n$ by
induction on $n$. $\EFR_0$ follows from the cancellation property. 
For $n = m + 1$ we distinguish
three cases.

%\\[.6mm]
Case $a_0 = 0$: Then $0_{a^2_0+\cdots + a^2_{m+1}}\cdot a_0=0$.

%\\[.6mm]
Case $a_i = 0$ for some $0 < i\leq m+1$: Without loss of generality we may assume 
that $i = m+1$.
Then $a^2_0+\cdots + a^2_{m+1}= a^2_0+\cdots + a^2_m$. Hence 
\[
0_{a^2_0+\cdots + a^2_{m+1}}\cdot a_0=0_{a^2_0+\cdots + a^2_{m}}\cdot a_0=0
\]
by the induction hypothesis.

%\\[.6mm]
Case $a_i \neq 0$ for all $0\leq i\leq m+1$: We obtain a contradiction as follows. 
Since $a^2_0+\cdots + a^2_{m+1} =0$
we have $\sg( a^2_0+\cdots + a^2_{m+1})= \sg(0) = 0$. On the other hand, it follows 
from the cancellation
property and \eqref{S7} that $\sg(a^2_i) = 1$ for every $0\leq i\leq m + 1$ and hence 
$\sg( a^2_0+\cdots + a^2_{m+1})= 1$
by \eqref{S6}.
\end{proof}

We can represent signed meadow terms in a standard way similar to ordinary meadow
terms.
\begin{definition}
A term $P$ over $\Sigma_{\textit{ms}}$ is a \emph{Signed Standard Meadow Form
(SSMF)} if, for some $n\in\Nat$, $P$ is an \emph{SSMF of level $n$}.
SSMFs of level $n$ are defined as follows:

%\\[1mm]
%\begin{description}
%\item[\text{\textit{SSMF of level $0$}}] each expression of the form 
SSMF of level $0$: each expression of the form 
$s\cdot t^{-1}$ 
with  $s$ and $t$ $\Sigma_{\textit{fs}}$-terms,
%\item[\textit{SSMF of level $n+1$}] each expression of the form

%\\[.6mm]
SSMF of level $n+1$: each expression of the form
\[0_t\cdot P+1_t\cdot Q\]
with  $t$  a $\Sigma_{\textit{fs}}$-term and $P$ and $Q$  SSMFs of 
level $n$.
%\end{description}
\end{definition}

\begin{theorem}
\label{SSMF}
For each term $s$ over $\Sigma_{ms}$ there exists  an SSMF $s'$
with the same variables such that  $\Md  + \SA\vdash s=s'$.
\end{theorem}

\begin{proof}
In~\cite{BBP2013} a proof is given for $\Sigma_{m}$-terms by structural induction. 
It thus   suffices to prove that the set of SSMFs is closed under $\sg$. We employ 
induction on level height $n$. If the SSMF $t$ has level 0 then 
$t=s_1\cdot {s_2}^{-1}$ for $\Sigma_{\textit{fs}}$-terms $s_1,s_2$. Then 
\[ 
\sg(t)=\sg(s_1\cdot s_2^{-1})= \sg(s_1)\cdot \sg(s_2^{-1})=\sg(s_1) \cdot 
\sg(s_2)=\sg(s_1\cdot s_2)\cdot 1^{-1}
\]
by axiom \eqref{S4} and \eqref{S5}. Assume $t=0_{s}\cdot P + 1_s \cdot Q$ for some 
$\Sigma_{\textit{fs}}$-term $s$ and SSMFs $P,Q$ of level $n$. Then 
$0_s\cdot t=0_s\cdot P$ and $1_s\cdot t=1_s\cdot Q$. 
Hence
\begin{align*}
\sg(t)&=(0_s + 1_s)\cdot \sg(t)\\
&=0_s\cdot \sg(t) + 1_s\cdot\sg(t)\\
&=\sg(0_s\cdot t) + \sg(1_s\cdot t)\\
&=0_s\cdot \sg(P) + 1_s \cdot \sg(Q)
\end{align*}
by \eqref{S1}, \eqref{S2} and \eqref{S5}.
\end{proof}

As in the case of ordinary meadow equations, it follows that every signed meadow 
equation has a first-order representation over the signature of signed fields. 
In the corollary and the proposition below we again apply freely $\IL$ since we 
will use these results only in the context of cancellation meadows.
\begin{corollary}\label{phi2}
For each equation $s=t$ over $\Sigma_{ms}$ there exists a quantifier-free 
first-order formula $\phi(s,t)$ over $\Sigma_{\textit{fs}}$ with the same 
variables such that 
\[\Md  +\SA + \IL\vdash s=t \leftrightarrow \phi(s,t).\]
\end{corollary}

Given a signed cancellation meadow, we shall consider it ordered by the order 
induced by the sign function. This can be expressed by the axiom
\begin{equation}
x < y \leftrightarrow \sg (y-x)=1 \tag{\SO\,}.
\end{equation}

\begin{proposition}\label{mainprop}
\mbox{}
\begin{enumerate}
\item $\Md  + \IL + \SA + \SO\vdash OF$, and
\item for all terms $s,t$ over $\Sigma_{\textit{fs}}$ there exists a formula 
$\psi(s,t)$ over $(\Sigma_f,<)$ with the same free variables such that
\[
\Md  + \IL +  \SA + \SO\vdash s=t \leftrightarrow \psi(s,t).
\]
\end{enumerate}
\end{proposition}
\begin{proof}
\begin{enumerate}
\item We prove the derivability of the axioms given in Table~\ref{orderedfield}.

\eqref{OF1}: Assume $x\neq 0$. Then $1_x=1$. Therefore
\begin{align*}
(1-\sg(x))\cdot (1+\sg(x)) &=1_x\cdot (1-\sg(x)^2)\\
&= 1_x - \sg(x)^2 &\text{by \eqref{S9}}\\
&= 1_x - 1_x & \text{by \eqref{S5}, \eqref{S7}}\\
&= 0.
\end{align*}
Hence $\sg(x) = 1$ or $\sg(x) = -1$ by cancellation. 
If $\sg(x) =1$ then $0< x$; if $\sg(x)=-1$ then 
$1= -1\cdot \sg(x)=\sg(-1)\cdot \sg(x)=\sg(-x)$ and thus $x<0$.

\eqref{OF2}: Assume $x<y$. Then $\sg(y-x)=1$. Thus, if $y<x$ then $\sg(x-y)=\sg(y-x)$ and 
so $0=\sg(0)=\sg((y-x)+(x-y))=1$ by \eqref{S6}. Likewise, if $x=y$, then 
$0=\sg(x-x)=\sg(y-x)=1$. Thus $\neg( y<x \ \vee\ x=y)$.

\eqref{OF3}: If $x<y$ then $1=\sg(y-x)=\sg((y+z)-(x+z))$. Hence $x+z < y+z$.

\eqref{OF4}: If $x<y$ and $0<z$ then $\sg(y-x)=1$ and $\sg(z)=1$. Thus 
$1=1\cdot 1= \sg(y-x)\cdot \sg(z) = \sg((y-x)\cdot z)=\sg(y\cdot z - x\cdot z)$, 
i.e.\ $x\cdot z < y\cdot z$.

\item Let $t$ be a $\Sigma_{fs}$-term, $\textit{Var}\,(t)$ be  the set of variables 
occurring in $t$ and $x\not\in \textit{Var}(t)$.
By structural induction, we will first construct a formula $\gamma(x,t)$ over 
$(\Sigma_f,<)$ with free variables 
$\textit{Var}\,(t)\cup \{x\}$ such that for all $(\Sigma_f,\sg)$-terms $s$, 
$M'\models s=t \leftrightarrow \gamma(s,t)$. 
\begin{enumerate}
\item $t=0$, $t=1$ or $t=y$ ($y\not \equiv x$):   $\gamma(x,t)\equiv x=t$,
\item $t\equiv -t'$: $\gamma(x,t)\equiv \forall z (\gamma(z,t') \rightarrow x=-z)$ 
with $z\not \in \textit{Var}\,(t')\cup \{x\}$,
\item $t\equiv t_1\diamond t_2$ with $\diamond\in \{ +,\cdot\}$: 
$\gamma(x,t)\equiv \forall y \forall z (\gamma(y,t_1) \ \wedge\ \gamma(z, t_2) 
\rightarrow x=y \diamond z)$ with $y, z\not \in \textit{Var}\,(t_1\diamond 
t_2)\cup \{x\}$,
\item $t\equiv \sg(t')$: 
\begin{eqnarray*}
\gamma(x,t)&\equiv& (x=0 \rightarrow \gamma(x,t')) \ \wedge\ \\
&&(x\neq 0 \rightarrow \forall z (\gamma(z,t') \rightarrow \\
&&\qquad(( 0<z \ \wedge\  x = 1)\ \vee\  (z<0 \ \wedge\ x = -1))))
\end{eqnarray*}
with $z\not \in \textit{Var}\,(t')\cup \{x\}$.
\end{enumerate}
We prove that $\gamma(x,\sg(t'))$ meets the requirements. 
$\gamma(x,\sg(t') )$ is a formula over $(\Sigma_f,<)$
since $\gamma(x,t')$ is so. Moreover, the free variables in $\gamma(x,\sg(t') )$ 
are $\textit{Var}\,(t')\cup \{x\}=\textit{Var}\,(\sg(t'))\cup \{x\}$. Now assume 
$s=\sg(t')$. If $s=0$ then $\sg(t')=0$. If $t'\neq 0$ we can argue as in (1) in the 
case of \eqref{OF1} to derive a contradiction. Thus $t'=0$. Hence $s=t'$ and therefore 
$\gamma(s,t')$. To prove the second conjunct assume $s\neq 0$ and $z=t'$. 
Then $z\neq 0$ and hence $0<z$ or $z<0$ by \eqref{OF1}. If $0<z$ then $s=\sg(z)=1$; 
likewise, if $z<0$ then $\sg(-z)=1$ and hence $s=\sg(z)=-1$. 
$\gamma(s,\sg(t')) \rightarrow s=\sg(t')$ follows in a similar fashion. 
Now observe that 
\[
M'\models s=t \leftrightarrow \forall x\forall y (( x=s \ \wedge\ y=t) 
\rightarrow x=y)
\]
with $x,y \not\in \textit{Var}\,(s)\cup \textit{Var}\,(t)$. Thus we can take 
\[
\psi(s,t)\equiv \forall x\forall y (( \gamma(x,s) \ \wedge\ \gamma(y,t)) 
\rightarrow x=y).
\]
\end{enumerate}
\end{proof}

We slightly modify the proof of Theorem~\ref{complete1} in order to obtain:

\begin{theorem}\label{main}
Let $s,t$ be $\Sigma_{ms}$-terms. Then 
\[
\Md + \SA\vdash s=t \quad  \text{ if and only if } \quad
(\mathbb{R}_0,\sg)\models s=t.
\]
\end{theorem}
\begin{proof}
Soundness is again immediate. 

Assume that $s=t$
is not derivable from $\Md +\SA$. By Corollary~\ref{basis2} we now have a counter 
model $M$ which is a signed cancellation meadow.
Since $M$ is formally real by Proposition~\ref{efr}, it is orderable. 
We consider the particular order given by 
\[x< y \leftrightarrow \sg(y-x) = 1.\]
Then $(M,<)$ is an ordered field by Proposition~\ref{mainprop}.1. 
By the Artin-Schreier Theorem $(M,<)$ has  real closure $(\widehat{M},<)$ whose 
ordering is an extension of the given ordering. 
We find that $(\widehat{M},<)\not \models s=t$. By Corollary~\ref{phi2}, $s=t$ is 
equivalent to a $(\Sigma_f,\sg)$-formula 
$\phi(s,t)$. Every equation $s'=t'$ in $\phi(s,t)$ can be replaced by a 
$(\Sigma_f,<)$-formula $\psi(s',t')$ by Proposition~\ref{mainprop}.2. Thus there 
exists a $(\Sigma_f,<)$-formula $\Gamma$ such that 
$(\widehat{M},<)\models s=t \leftrightarrow \Gamma$ and hence 
$(\widehat{M},<)|\Sigma_f\not \models  \Gamma$. Since the theory of real closed 
fields coincides  with the theory of the real numbers  over the signature 
$(0,1,-,  +, \cdot,  <)$, we can conclude that $\mathbb{R}_0\not \models\Gamma$. 
Therefore $(\mathbb{R}_0, \sg)\not \models s=t$.
\end{proof}

\begin{corollary}
If an open quantifier free formula $\phi(x_1,...,x_k)$ is satisfiable in some signed 
cancellation meadow then that formula is satisfiable in the signed meadow of reals.
\end{corollary}
\begin{proof}
This is the same fact in another wording. If $\phi(x_1,...,x_k)$ is not satisfiable 
in the signed reals then  $\neg \phi(x_1,...,x_k)$
is valid in the reals and by the preceding theorem it must be derivable from 
$\Md +\SA$, whence it is true in all cancellation meadows
in contradiction with the assumption.
\end{proof}

\begin{remark}
It should be noted that whether or not the equational theory of signed rationals can 
be finitely 
axiomatized is still an open problem. The case of the reals is much simpler indeed. 
\end{remark}

\section{The meadow of complex numbers}
Rational complex numbers have been studied before in \cite{BT06}. 
In this section, we generalize that approach in order to give a completeness result for  
arbitrary complex meadow terms.

We extend the signature $\Sigma_m$ of  meadows to the signature of complex numbers 
$(\Sigma_m,i, \bar{\ })$ where 
$i$ is the imaginary unit and $\bar{\ }$ is the unary operation of  complex conjugation. 
The  defining  equations  are listed in Table~\ref{conjugation} and 
Table~\ref{absolutevalue}.  
\eqref{CC0}|\eqref{CC8} are the usual equalities for complex numbers and (CC9) ensures 
propagation;  $\SSAV$ expresses the fact that any  sum of squares of absolute values of 
complex numbers cannot yield -1. 
As a special instance of \SSAV\  we obtain the set of axioms \CO\  for meadows of 
characteristic 0 (see also Table \ref{CO}):
\begin{align}
1_{\underline{n+1}}&= 1_{1+ 1\cdot \overline{1} + \cdots + 1\cdot \overline{1}}=1  
&\tag{C0$_n$}
\end{align}

\begin{table}
\centering
\hrule
\begin{align}
	\overline{0} &= 0&\tag{CC0}\label{CC0}\\
	\overline{1} &= 1&\tag{CC1}\label{CC1}\\
	\overline{i} &= -i&\tag{CC2}\label{CC2}\\
	\overline{-x} &= -\overline{x}&\tag{CC3}\label{CC3}\\
	\overline{x+y}     &= \overline{x}+\overline{y}&\tag{CC4}\label{CC4}\\
	\overline{x\cdot y}     &= \overline{x}\cdot \overline{y}&\tag{CC5}\label{CC5}\\
	\overline{x^{-1}} &= \overline{x}^{-1}&\tag{CC6}\label{CC6}\\
	\overline{\overline{x}} &= x&\tag{CC7}\label{CC7}\\
	i\cdot i & =  -1&\tag{CC8}\label{CC8}\\
	1_{\overline{x}}&=1_x&\tag{CC9}\label{CC9}
\end{align}
\hrule
\caption{The set \CC\ of axioms for complex conjugation}
\label{conjugation}
\end{table}

\begin{table}
\centering
\hrule
\begin{align}
1_{1+ x_0\cdot \overline{x_0} + \cdots +x_n\cdot \overline{x_n}}&=1&\tag{SSAV$_n$}\label{AEFRN}
\end{align}
\hrule
\caption{The set $\SSAV$ of axioms for the absolute value of complex numbers}
\label{absolutevalue}
\end{table}

In the sequel we will use the derived unary operators $re(\_)$ and $im(\_)$|the real 
part and the imaginary part of a complex 
number|given in Table \ref{reim}. Here we write $\frac{s}{t}$ for $s\cdot t^{-1}$. 
$re(\_)$ and $im(\_)$ enjoy a couple of nice properties which are listed in Table~\ref{RI}.

\begin{table}
\centering
\hrule
\begin{align}
re(x)&=\frac{1}{\underline{2}}\cdot (x + \overline{x})\notag\\
im(x)&=-\frac{i}{\underline{2}}\cdot (x - \overline{x})\notag
\end{align}
\hrule
\caption{The derived operators $re(\_)$ and $im(\_)$}
\label{reim}
\end{table}

\begin{table}
\centering
\hrule
\begin{align}
	x&= re(x) + im(x) \cdot i&\tag{RI0}\label{RI0}\\
	re(x)&=\overline{re(x)}&\tag{RI1}\label{RI1}\\
	im(x)&=\overline{im(x)}&\tag{RI2}\label{RI2}\\
	re(re(x))&=re(x)&\tag{RI3}\label{RI3}\\
	re(im(x))&=im(x)&\tag{RI4}\label{RI4}\\
	im(re(x))&=0&\tag{RI5}\label{RI5}\\
	im(im(x))&=0&\tag{RI6}\label{RI6}\\
	re(0)&=0 &\tag{RI7}\label{RI7}\\
	re(1)&=1& \tag{RI8}\label{RI8}\\
	re(i)&=0 &\tag{RI9}\label{RI9}\\
	re(-x)&=-re(x) & \tag{RI10}\label{RI10}\\
	re(x+y)&=re(x) + re(y)& \tag{RI11}\label{RI11}\\
	re(x\cdot y)&=re(x)\cdot re(y) - im(x)\cdot im(y)& \tag{RI12}\label{RI12}\\
	re(x^{-1})&=re(x)\cdot (re(x)\cdot re(x) + im(x)\cdot im(x))^{-1} & \tag{RI13}\label{RI13}\\
	re(\overline{x})&=re(x)& \tag{RI14}\label{RI14}\\
	im(0)&=0 &\tag{RI15}\label{RI15}\\
	im(1)&=0& \tag{RI16}\label{RI16}\\
	im(i)&=1 &\tag{RI17}\label{RI17}\\
	im(-x)&=-im(x) & \tag{RI18}\label{RI18}\\
	im(x+y)&=im(x) + im(y)& \tag{RI19}\label{RI19}\\
	im(x\cdot y)&=re(x)\cdot im(y) + im(x)\cdot re(y)& \tag{RI20}\label{RI20}\\
	im(x^{-1})&=-im(x)\cdot (re(x)\cdot re(x) + im(x)\cdot im(x))^{-1} & \tag{RI21}\label{RI21}\\
	im(\overline{x})&=-im(x)& \tag{RI22}\label{RI22}
\end{align}
\hrule
\caption{Properties of  the real part and the imaginary part of complex numbers}
\label{RI}
\end{table}

\begin{proposition}\label{RIi}
For  $i\in \{0, \ldots , 22\}$, 
\[
\Md + \CC + \SSAV \vdash \text{RIi}
\]
\end{proposition}
\begin{proof}
We prove \eqref{RI0}, \eqref{RI1}, \eqref{RI2}, \eqref{RI13} and \eqref{RI21}.
\begin{itemize}
\item (RI0): 
\[
\begin{array}{rcll}
re(x) + im(x)\cdot i&=& \frac{1}{2}\cdot (x +\overline{x})- \frac{i}{\underline{2}}\cdot (x -\overline{x})\cdot i\\[0.1cm]
&=& \frac{1}{\underline{2}}\cdot (x +\overline{x}) +\frac{1}{\underline{2}}\cdot (x -\overline{x}) &\text{ by \eqref{CC8}}  \\[0.1cm]
&=& \frac{1}{\underline{2}}\cdot x +\frac{1}{\underline{2}}\cdot x  \\[0.1cm]
&=&\underline{2}\cdot \underline{2}^{-1} \cdot x \\
&=& x & \text{ by  (C0$_1$)}
\end{array}
\]

\item (RI1):
\[
\begin{array}{rcll}
\overline{re(x)}&=& \overline{\frac{1}{\underline{2}}(x + \overline{x})}\\[0.1cm]
&=& \frac{1}{\underline{2}}\overline{(x + \overline{x})}& \text{ by (CC1), (CC4), (CC5), (CC6)}\\[0.1cm]
&=& \frac{1}{\underline{2}}(\overline{x} + \overline{\overline{x}})& \text{ by (CC4)}\\[0.1cm]
&=& \frac{1}{\underline{2}}(\overline{x} + x)& \text{ by (CC7)}\\[0.1cm]
&=& \frac{1}{\underline{2}}(x+\overline{x})&\\
&=& re(x)
\end{array}
\]
\item (RI2):
\[
\begin{array}{rcll}
\overline{im(x)}&=& \overline{-\frac{i}{\underline{2}}(x - \overline{x})}\\[0.1cm]
&=& \frac{i}{\underline{2}}\overline{(x - \overline{x})}& \text{ by (CC1), (CC2), (CC3), (CC4), (CC5), (CC6)}\\[0.1cm]
&=& \frac{i}{\underline{2}}(\overline{x} - \overline{\overline{x}})& \text{ by (CC3), (CC4)}\\[0.1cm]
&=& \frac{i}{\underline{2}}(\overline{x} - x)& \text{ by (CC7)}\\[0.1cm]
&=&- \frac{i}{\underline{2}}(x-\overline{x})&\\
&=& im(x)
\end{array}
\]
\item (RI13), (RI21): Observe that
\[ 
\begin{array}{rcll}
t^{-1} &=& t^{-1}\cdot 1_{t^{-1}}\\
 &=& t^{-1}\cdot 1_t\\
&=& t^{-1}\cdot 1_{\overline{t}}&\text{ by (CC9)}\\
&=& t^{-1}\cdot \overline{t}^{-1}\cdot \overline{t}&\text{ by (CC5), (CC6)} \\
&=& (t \cdot \overline{t})^{-1}\cdot \overline{t}\\
&=& (re(t)\cdot re(t)+im(t)\cdot im(t))^{-1}\cdot (re(t) - im(t)\cdot i)&\text{ by (RI0), (RI1), (RI2),}\\
&&&\text{ (CC2)}
\end{array}
\]
Thus
\[
\begin{array}{rcll}
re(t^{-1})&=& \frac{1}{\underline{2}}(t^{-1} + \overline{t^{-1}})\\[0.1cm]
&=&\frac{1}{\underline{2}}((re(t)\cdot re(t)+im(t)\cdot im(t))^{-1}\cdot (re(t)-  im(t)\cdot i)\\
&&+(re(t)\cdot re(t)+im(t)\cdot im(t))^{-1}\cdot (re(t) + im(t)\cdot i))&\text{ by (RI1), (RI2),}\\
&&&\hspace{-8mm}\text{ (CC2), (CC3), (CC4)}\\ 
&=& re(t) \cdot (re(t)\cdot re(t)+im(t)\cdot im(t))^{-1}
\end{array}
\]
Likewise
\[
\begin{array}{rcll}
im(t^{-1})&=&- \frac{i}{\underline{2}}(t^{-1} - \overline{t^{-1}})\\[0.1cm]
&=&-\frac{i}{\underline{2}}((re(t)\cdot re(t)+im(t)\cdot im(t))^{-1}\cdot (re(t)-  im(t)\cdot i)\\
&&-(re(t)\cdot re(t)+im(t)\cdot im(t))^{-1}\cdot (re(t) + im(t)\cdot i))&\\
&&\hspace{3 cm}\text{ by (RI1), (RI2), (CC2), (CC3), (CC4)}\\ 
&=& -im(t) \cdot (re(t)\cdot re(t)+im(t)\cdot im(t))^{-1}
\end{array}
\]

\end{itemize}
\end{proof}
The real and the imaginary part of a complex number can be represented by so-called 
\emph{real forms}.
\begin{definition}
The set of \emph{real forms} is defined inductively as follows.
\begin{enumerate}
\item  $0$ and $1$ are real forms,
\item $re(x)$ and $im(x)$ are real forms for every variable $x$,
\item if $s$ and $t$ are real forms then so are $-s$, $s^{-1}$, $s+t$ and $s\cdot t$.
\end{enumerate}

\end{definition}
\begin{lemma}\label{realterms}
For each  $t\in (\Sigma_m, i, \bar{\ })$ there exist real forms  $t_1,t_2$ such that 
\[
\Md + \CC \vdash re(t)=t_1 \text{ and } \Md + \CC  \vdash im(t)=t_2.
\]
\end{lemma}
\begin{proof}
This follows by an easy induction on $t$.
\end{proof}
Since real and imaginary parts of complex numbers are independent, 
we can interpret ordinary meadow terms as complex terms while retaining provable equality.

\begin{proposition}\label{meadowterms}
Let $s,t$ be  $\Sigma_m$-terms with free variables among $\{x_0, \ldots, x_n,$ $ y_0, \ldots y_n\}$. For fresh variables $\{z_0, \ldots , z_n\}$, 
\[ t^*\equiv t[x_0, \ldots, x_n:=re(z_0), \ldots , re(z_n)][y_0, \ldots, y_n:=im(z_0), \ldots , im(z_n)]\]
and 
\[ s^*\equiv s[x_0, \ldots, x_n:=re(z_0), \ldots , re(z_n)][y_0, \ldots, y_n:=im(z_0), \ldots , im(z_n)]\]
 are real forms such that 
\[
\Md + \AEFR \vdash s=t  \Rightarrow \Md  +\CC + \SSAV \vdash s^* =t^*
\]
Here, we adopt the notation $[v_1, \ldots , v_n:=r_1, \ldots , r_n]$ 
for the substitution $\sigma$ with $\sigma(v_i)=r_i$ 
for $0\leq i\leq n$ and $\sigma(v)=v$ for all $v\not\equiv v_i$, $0\leq i\leq n$.
\end{proposition}

\begin{proof}
Suppose $ \Md  +\CC + \SSAV\not \vdash s^* =t^*$. Then there exists a 
$(\Sigma_m, i,\bar{\  })$-structure 
$\mathcal{M}=(U, 0,1, i, -, + ,\cdot , \  ^{-1}, \bar{\ })$ 
with $\mathcal{M}\not\models s^* =t^*$. 
Let 
\[\mathcal{M}'=(U', 0_{\mathcal{M}'},1_{\mathcal{M}'}, -_{\mathcal{M}'}, 
+_{\mathcal{M}'} ,\cdot_{\mathcal{M}'} , \  ^{-1_{\mathcal{M}'}})\]
be the $\Sigma_m$-structure obtained from $\mathcal{M}$ by 
stipulating $U'=\{re(u) \mid u\in U\}$, $0_{\mathcal{M}'}=0$, 
$1_{\mathcal{M}'}=1$, and for $u, u_0,u_1\in U'$,
\[
\begin{array}{rcl}
-_{\mathcal{M}'} (u)&=& -u\\
+_{\mathcal{M}'} (u_0 ,u_1)&=&u_0+u_1\\
\cdot_{\mathcal{M}'} (u_0, u_1)&=&u_0 \cdot u_1\\
(u)^{-1_{\mathcal{M}'}}&=&u^{-1}.
\end{array}
\]
Then $\mathcal{M}'$ is well-defined: 
$0_{\mathcal{M}'}, 1_{\mathcal{M}'} \in U'$ by (RI7), (RI8), and e.g.\ since 
\begin{align*}
re(re(u)^{-1})&= re(re(u))\cdot (re(re(u))\cdot re(re(u)) + im(re(u))\cdot im(re(u)))^{-1}\\
&\hspace{9 cm} \text{ by (RI13)}\\
&=re(u)\cdot (re(u)\cdot re(u) )^{-1} \text{ by (RI3), (RI5)}\\
&= re(u)^{-1}
\end{align*}
it follows that $(u)^{-1_{\mathcal{M}'}}\in U'$ for every $u\in U'$.
Moreover, since $\mathcal{M}\not\models s^* =t^*$, $\mathcal{M}'\not\models s=t$. 
Clearly,  $\mathcal{M}'\models \Md$ and as 
\[
1_{1+ re(x_0)\cdot re(x_0)+ \cdots +re(x_n)\cdot re(x_n)}= 1_{1+ re(x_0)\cdot \overline{re(x_0)}+ \cdots +re(x_n)\cdot \overline{re(x_n)}}=1
\]
also $\mathcal{M}'\models \AEFR$.
Thus $\Md + \AEFR \not \vdash s=t $. 
\end{proof}

We can now apply our previous completeness result in order to prove the 
axiomatization of the complex numbers to be complete.

\begin{theorem}\label{complex}
Let $s,t$ be $(\Sigma_{m}, i, \bar{\ })$-terms. Then 
\[
\Md + CC + \SSAV \vdash s=t \quad  \text{ if and only if  } \quad (\mathbb{C}_0,i,\bar{\ })\models s=t.
\]
\end{theorem}
\begin{proof}
Soundness is again immediate. 

Assume $(\mathbb{C}_0,i,\bar{\ })\models s=t$. 
Then $(\mathbb{C}_0,i,\bar{\ })\models s-t=0$. By Lemma \ref{realterms} we can pick real 
forms $r_1$, $r_2$ such that 
\[
\Md + CC + \SSAV \vdash re(s-t)=r_1 \text{ and } \Md + CC + \SSAV \vdash im(s-t)=r_2.
\]
Then
\[
(\mathbb{C}_0,i,\bar{\ })\models re(s-t)=r_1 \text{ and } 
(\mathbb{C}_0,i,\bar{\ })\models im(s-t)=r_2.
\]
From (RI7) and (RI15) it follows that 
\[
(\mathbb{C}_0,i,\bar{\ })\models r_1=0 \text{ and }(\mathbb{C}_0,i,\bar{\ })\models r_2=0.
\]
Without loss of generality, we may assume that the free variables of $r_1$ and $r_2$ are 
amongst $\{z_0, \ldots, z_n\}$. Now choose fresh variables 
$\{x_0, \ldots, x_n, y_0, \ldots y_n\}$ and  consider the meadow terms $u_1$, 
$u_2$ obtained from $r_1, r_2$ by replacing occurrences of a subterm of the form 
$re(z_i)$ by the variable $x_i$ and occurrences of subterms of the form 
$im(z_i)$ by the variable $y_i$. Then 
\[
\mathbb{R}_0\models u_1=0 \text{ and }\mathbb{R}_0\models u_2=0
\]
and hence 
\[
\Md + \EFR \vdash u_1=0 \text{ and } \Md + \EFR \vdash u_2=0
\]
by Theorem \ref{complete1}. Whence 
\[
\Md + \AEFR \vdash u_1=0 \text{ and } \Md + \AEFR \vdash u_2=0
\]
by Theorem \ref{alternative}. Since $r_1\equiv u_1^*$ and $r_2\equiv u_2^*$, 
it follows  from Proposition \ref{meadowterms} that 
\[
\Md + CC + \SSAV \vdash r_1=0 \text{ and } \Md + CC + \SSAV \vdash r_2=0.
\]
Therefore 
\[
\Md + CC + \SSAV \vdash s-t = re(s-t) + im(s-t)\cdot i=0,
\]
and thus $\Md + CC + \SSAV \vdash s=t$.
\end{proof}
\begin{remark}
One obtains a finite axiomatization of the complex numbers by replacing \SSAV\   
by the eight axioms \eqref{S1} -- \eqref{S3}, \eqref{S*4} -- \eqref{S*8} for the sign function 
given in Table \ref{add}. Note that $\sg$ differs from the usual generalizations of the sign function on $\mathbb{R}$ to $\mathbb{C}-\{0\}$ given by $sgn(z) = z/|z|$|where $|z|$ is the absolute value of $z$|and
\begin{equation}
\nonumber
csgn(z)=
\begin{cases}
1&\text{if $re(z)>0$},\\
-1&\text{if $re(z)<0$},\\
sgn(im(z))&\text{if $re(z)=0$}.
\end{cases}
\end{equation}
That $\Md + \CC  + \SA^*$ is indeed a complete axiomatization of the meadow of 
complex numbers can be seen by redoing  the proof given above. Observe that (C0$_1$)|which is used in Proposition \ref{RIi}|follows from \SA* by
\[
1_{\underline{2}}=\sg( 1_{\underline{2}})=\sg(re(1))=\sg(1)=1.
\]

\begin{table}
\centering
\hrule
\begin{align}
\sg(1_x)&=1_x&\tag{\ref{S1}}\\
\sg(0_x)&=0_x&\tag{\ref{S2}}\\
\sg(-1)&=-1&\tag{\ref{S3}}\\
\sg(re(x)^{-1})&=\sg(re(x))&\tag{S*4}\label{S*4}\\
\sg(re(x)\cdot re(y))&=\sg(re(x))\cdot \sg(re(y))&\tag{S*5}\label{S*5}\\
0_{\sg(re(x))-\sg(re(y))}\cdot (\sg(re(x)+re(y))-\sg(re(x)))&=0&\tag{S*6}\label{S*6}\\
\sg(x)&=\sg(re(x))&\tag{S*7}\label{S*7}\\
\overline{\sg(x)}&=\sg(x)&\tag{S*8}\label{S*8}
\end{align}
\hrule
\caption{The set \SA*\ of axioms for the sign function for complex numbers}
\label{add}
\end{table}.
\end{remark}

\section*{References}
\bibliographystyle{elsarticle-num}

\appendix
\renewcommand*{\thesection}{\Alph{section}}

\section{Appendix}
In this appendix we will give a syntactic proof of Proposition~\ref{efr} without 
the use of $\IL$. We obtained this proof from a proof relying on  $\IL$ by applying 
the $\IL$-elimination steps suggested in the proof of Theorem~\ref{st}. 
We first list some properties
of pseudo ones and pseudo zeros in Table~\ref{pc}.
\begin{table}
\centering
\hrule
\begin{align}
0_t\cdot 0_t&=0_t&\tag{PC1}\label{PC1}\\
0_{t^2}&= 0_t&\tag{PC2}\label{PC2}\\
0_t\cdot t&= 0&\tag{PC3}\label{PC3}\\
0_s\cdot 0_{s+t} &= 0_s\cdot 0_t&\tag{PC4}\label{PC4}\\
1_t\cdot 1_t&=1_t&\tag{PC5}\label{PC5}\\
1_{t^2}&=1_t&\tag{PC6}\label{PC6}\\
1_t\cdot t&=t&\tag{PC7}\label{PC7}\\
1_s\cdot 1_t&=1_{s\cdot t}&\tag{PC8}\label{PC8}
\end{align}
\hrule
\caption{Some properties of pseudo constants for $\Sigma_{ms}$-terms $s,t$}
\label{pc}
\end{table}

We derive \eqref{PC4}: Since 
\begin{align*}
0_s \cdot (s+t)
&=0_s\cdot s +0_s\cdot t\\
&= 0_s\cdot t&\text{by~\eqref{PC3}}
\end{align*}
we have 
\begin{align*}
0_s^{-1} \cdot (s+t)^{-1}
&=(0_s \cdot (s+t))^{-1}\\
&= (0_s\cdot t)^{-1}\\
&=0_s^{-1}\cdot t^{-1}
\end{align*}
and hence 
\begin{align*}
0_s\cdot (s+t)^{-1}
&= 0^2_s\cdot 0_s^{-1} \cdot (s+t)^{-1}\\
&=0^2_s\cdot 0_s^{-1}\cdot t^{-1}\\
&=0_s\cdot t^{-1}.
\end{align*}
Thus
\begin{align*}
0_s\cdot 0_{s+t}
&= 0_s - 0_s\cdot (s+t)\cdot (s+t)^{-1}\\
&=0_s - 0_s\cdot t^{-1} \cdot (s+t)\\
&=0_s - 0_s\cdot t \cdot t^{-1}\\
&=0_s\cdot 0_{t}.
\end{align*}
For the remaining identities see~\cite{BBP2013}. Moreover, we have the following 
useful identity.

\begin{lemma}\label{useful}
For all $n\in \mathbb{N}$, 
$\Md  +\SA \vdash 0_{x_0^2 + \cdots + x_n^2}\cdot 1_{x_0\cdot\ \cdots\ \cdot x_n}=0$.
\end{lemma}

\begin{proof}
We first prove by induction on $n$ that
\begin{align}
\Md +\SA \vdash 1_{x_0\cdot\ \cdots\ \cdot x_n}\cdot \sg(x_0^2 + \cdots + x_n^2)= 
1_{x_0\cdot\   \cdots\  \cdot x_n}&\tag{\dag}
\end{align}
For $n=0$ observe that 
$
1_x \cdot \sg(x^2)=1_x\cdot 1_x=1_x
$
by \eqref{S7} and \eqref{PC5}. Suppose $n=m+1$. By the induction hypothesis we have 
\[
1_{x_0\cdot\ \cdots\ \cdot x_m}\cdot \sg(x_0^2 + \cdots + x_m^2)= 
1_{x_0\cdot\   \cdots\  \cdot x_m}
\]
and thus $1_{x_0\cdot\ \cdots\ \cdot x_{m+1}}\cdot \sg(x_0^2 + \cdots + x_m^2)= 
1_{x_0\cdot\   \cdots\  \cdot  x_{m+1}}$ by \eqref{PC8}.
It follows that 
\begin{align*}
0 &= 1_{x_0\cdot\ \cdots\  \cdot x_{m+1}}\cdot \sg(x_0^2 + \cdots + x_m^2) - 
1_{x_0\cdot\   \cdots\  \cdot x_{m+1}}\\
&= 1_{x_0\cdot\ \cdots\  \cdot x_{m+1}}\cdot \sg(x_0^2 + \cdots + x_m^2) - 
1_{x_0\cdot\   \cdots\  \cdot x_m} \cdot 1_{x_{m+1}}&\text{by \eqref{PC8}}\\
&=1_{x_0\cdot\ \cdots\  \cdot x_{m+1}}\cdot \sg(x_0^2 + 
\cdots + x_m^2) - 1_{x_0\cdot\   \cdots\  \cdot x_m} \cdot 1_{x_{m+1}}
\cdot 1_{x_{m+1}}\!\!\!&\text{by \eqref{PC5}}\\
&=1_{x_0\cdot\ \cdots\  \cdot x_{m+1}}\cdot \sg(x_0^2 + \cdots + x_m^2) - 
1_{x_0\cdot\   \cdots\   \cdot x_{m+1}}\cdot 1_{x_{m+1}}  &\text{by \eqref{PC8}}\\
&=1_{x_0\cdot\ \cdots\  \cdot x_{m+1}}\cdot \sg(x_0^2 + \cdots + x_m^2) - 
1_{x_0\cdot\   \cdots\   \cdot x_{m+1}}\cdot  \sg(x^2_{m+1})& \text{by \eqref{S7}}\\
&=1_{x_0\cdot\ \cdots\ \cdot x_{m+1}}\cdot (\sg(x_0^2 + \cdots + x_m^2) -  
\sg(x^2_{m+1})).
\end{align*}
Therefore
\[
1_{x_0\cdot\ \cdots\  \cdot x_{m+1}}\cdot 0_{\sg(x_0^2 + \cdots + x_m^2) -  
\sg(x^2_{m+1})}= 1_{x_0\cdot\ \cdots\  \cdot x_{m+1}}
\]
and hence by \eqref{S6}, \eqref{S7} and \eqref{PC5}
\begin{align*}
0&=1_{x_0\cdot \ \cdots\ \cdot  x_{m+1}}\cdot 0 \\
&= 1_{x_0\cdot \ \cdots\ \cdot   x_{m+1}} \cdot 0_{\sg (x_0^2+\cdots+ x_m^2) -
\sg(x_{m+1}^2)}  \cdot (\sg(x_0^2  +\cdots + x_{m+1}^2) -\sg(x_{m+1}^2))\\
&= 1_{x_0\cdot \ \cdots\ \cdot   x_{m+1}} \cdot 1_{x_0\cdot \ \cdots\ \cdot   
x_{m+1}}  \cdot (\sg(x_0^2 + \cdots  +x_{m+1}^2)  -1_{x_0\cdot \ \cdots x_{m+1}} )\\
&= 1_{x_0\cdot \ \cdots\ \cdot   x_{m+1}} \cdot  \sg(x_0^2 + \cdots  +x_{m+1}^2)  
-1_{x_0\cdot \ \cdots\ \cdot   x_{m+1}},
\end{align*}
i.e.\  
$1_{x_0\cdot \ \cdots\ \cdot   x_{m+1}} \cdot  \sg(x_0^2 + \cdots  +x_{m+1}^2) =
1_{x_0\cdot \ \cdots\ \cdot   x_{m+1}}$.

We then have
\begin{align*}
0&=0\cdot 1_{x_0\cdot \ \cdots\ \cdot   x_{n}}\\
&= \sg(0)\cdot 1_{x_0\cdot \ \cdots\ \cdot   x_{n}}\\
&=\sg(0_{x^2_0 + \cdots + x^2_n}\cdot (x^2_0 + \cdots + x^2_n))\cdot 1_{x_0
\cdot \ \cdots\ \cdot   x_{n}}& \text{ by \eqref{PC3}}\\
&=0_{x^2_0 + \cdots + x^2_n}\cdot \sg(x^2_0 + \cdots + x^2_n)\cdot 1_{x_0
\cdot \ \cdots\ \cdot   x_{n}}& \text{ by \eqref{S2}, \eqref{S5}}\\
&=0_{x^2_0 + \cdots + x^2_n}\cdot 1_{x_0\cdot \ \cdots\ \cdot   x_{n}}& 
\text{ by $(\dag)$.}
\end{align*}
\end{proof}

We now prove:

\begin{rev} 
For every $n\geq 0$, $\Md + \SA \vdash \EFR_n$.
\end{rev} 

\begin{proof}
We prove $\EFR_n$ by induction on $n$. For $n=0$ observe that 
$0_{x^2}\cdot x=0_x\cdot x =0$ by \eqref{PC2} and \eqref{PC3}.
For $n=m+1$ assume that $\Md +\SA \vdash 0_{y_0^2 + \cdots + y_m^2} \cdot y_0 = 0$ 
(IH). Then for all $0\leq i\leq m+1$
\begin{align}
0_{x_0^2 + \cdots + x_{m+1}^2} \cdot 0_{x_i}\cdot x_0 &= 0:&\tag{\ddag}
\end{align}
for $i=0$, this follows from $0_{x_0}\cdot x_0=0$ \eqref{PC3}; for $1\leq i\leq m+1$ 
we have  
\begin{align*}
0_{x_0^2 + \cdots + x_{m+1}^2} \cdot  0_{x_i} \cdot x_0
&= 0_{x_0^2 + 
\cdots + x_{m+1}^2} \cdot  0_{x_i^2}\cdot x_0&\text{by \eqref{PC2}}\\
&= 0_{x_i^2}\cdot  0_{x_0^2 + \cdots + x_{m+1}^2}\cdot x_0 \\
&= 0_{x_i^2}\cdot  0_{x_0^2 + \cdots + x_{i-1}^2 + x_{i+1}^2 + \cdots 
 + x_{m+1}^2+ x_i^2}\cdot x_0\!\!\!\\
&= 0_{x_i^2}\cdot  0_{x_0^2 + \cdots + x_{i-1}^2 + x_{i+1}^2 + \cdots 
 + x_{m+1}^2}\cdot x_0&\text{by \eqref{PC4}}\\
 &=0&\text{by (IH)}.
\end{align*}
Thus
 \begin{align*}
 &0_{x_0^2+\cdots + x_{m+1}^2}\cdot  x_0\\
 &=  0_{x_0^2 + \cdots + x_{m+1}^2} 
 \cdot (1_{x_0} + 0_{x_0}) \cdot x_0\\
 &=0_{x_0^2 + \cdots + x_{m+1}^2} \cdot 1_{x_0} \cdot x_0&\text{by ($\ddag$)}\\
 &= 0_{x_0^2 + \cdots + x_{m+1}^2} \cdot 1_{x_0}\cdot (1_{x_1} + 0_{x_1}) 
 \cdot x_0\\
 &= 0_{x_0^2 + \cdots + x_{m+1}^2} \cdot 1_{x_0\cdot x_1}  \cdot x_0
 &\text{by ($\ddag$) and \eqref{PC8}}\\
 &\vdots\\
 &=  0_{x_0^2 + \cdots + x_{m+1}^2} \cdot 1_{x_0\cdot\  \cdots\ \cdot x_m}\cdot 
 (1_{x_{m+1}} + 0_{x_{m+1}}) \cdot x_0\!\!\!\\ 
 &=  0_{x_0^2 + \cdots + x_{m+1}^2} \cdot 1_{x_0\cdot \ \cdots \ \cdot x_{m+1}}  
 \cdot x_0&\text{by ($\ddag$) and \eqref{PC8}}.
 \end{align*}
We may therefore conclude that 
\begin{align*}
&0_{x_0^2+\cdots  +x_{m+1}^2}\cdot x_0\\
&= 0_{x_0^2 + \cdots  
+x_{m+1}^2}\cdot(1_{x_0\cdot\ \cdots\  \cdot\ x_{m+1}} + 0_{x_0 \cdot\ 
\cdots\  \cdot  x_{m+1}})\cdot x_0\\
&=0_{x_0^2 + \cdots  +x_{m+1}^2}\cdot 1_{x_0\cdot\ \cdots\ \cdot\  x_{m+1}}
\cdot x_0 + 0_{x_0^2 + \cdots  +x_{m+1}^2}\cdot 0_{x_0\cdot\ \cdots\  
\cdot\ x_{m+1}} \cdot x_0 \\
&= 0_{x_0^2 + \cdots  +x_{m+1}^2}\cdot 0_{x_0\cdot\ \cdots\  \cdot\ x_{m+1}} 
\cdot x_0\hspace{26mm}\text{by Lemma~\ref{useful}}\\
&=0_{x_0^2 + \cdots  +x_{m+1}^2}\cdot 1_{x_0\cdot\ \cdots\  \cdot\ x_{m+1}} 
\cdot 0_{x_0\cdot\ \cdots\  \cdot\ x_{m+1}} \cdot x_0\hspace{4mm}\text{by the above 
identity} \\ 
&=0 \hspace{70.4mm} \text{ since $1_t\cdot 0_t=0$}.
\end{align*}
\end{proof}

\end{document}